\documentclass[a4paper,12pt,reqno]{amsart}
\usepackage[english]{babel}
\usepackage[T1]{fontenc}
\usepackage[left=1in,right=1in,top=1.2in,bottom=1.2in]{geometry}
\usepackage{times}
\usepackage{microtype}

\usepackage{commath,amssymb,amscd,mathrsfs,mathtools,dsfont,bbm,multirow,array,caption,comment,pifont}
\usepackage[all]{xy}
\usepackage{enumerate}
\usepackage{longtable}

\usepackage{cite}			%won't show labels for citation

\usepackage[usenames,dvipsnames,svgnames,table]{xcolor}
\definecolor{red}{RGB}{255,25,25}
\definecolor{blue}{RGB}{25,50,200}
\usepackage{tikz-cd}
\usepackage[pagebackref,linktocpage]{hyperref}

\hypersetup{
pdftitle={},				% title
pdfauthor={Hu--Truong},		% author
colorlinks=true,			% false: boxed links; true: colored links
linkcolor=red,			% color of internal links (change box color with linkbordercolor)
citecolor=MidnightBlue,	% color of links to bibliography
filecolor=magenta,		% color of file links
urlcolor=MidnightBlue	% color of external links
}

\usepackage{cleveref} % after hyperref!!

\newtheorem{theorem}{Theorem}%[section]
\crefname{theorem}{Theorem}{Theorems}
\newtheorem{lemma}[theorem]{Lemma}
\crefname{lemma}{Lemma}{Lemmas}

\crefname{proposition}{Proposition}{Propositions}

\crefname{prop}{Proposition}{Propositions}

\crefname{corollary}{Corollary}{Corollaries}

\crefname{cor}{Corollary}{Corollaries}

\crefname{conjecture}{Conjecture}{Conjectures}

\crefname{conj}{Conjecture}{Conjectures}
\newtheorem*{conj*}{Conjecture}
\crefname{conj}{Conjecture}{Conjectures}

\crefname{conjA}{Conjecture}{Conjecture}

\crefname{conjB}{Conjecture}{Conjecture}

\crefname{conjC}{Conjecture}{Conjecture}

\crefname{conjDk}{Conjecture}{Conjecture}

\crefname{conjD}{Conjecture}{Conjecture}

\crefname{conjH}{Conjecture}{Conjecture}

\crefname{conjGr}{Conjecture}{Conjecture}

\crefname{conjSGr}{Conjecture}{Conjecture}

\theoremstyle{definition}
\newtheorem{definition}[theorem]{Definition}
\crefname{definition}{Definition}{Definitions}

\crefname{defn}{Definition}{Definitions}

\crefname{example}{Example}{Examples}

\crefname{notation}{Notation}{Notation}
\newtheorem*{notation*}{Notation}
\crefname{notation}{Notation}{Notation}

\crefname{problem}{Problem}{Problems}

\crefname{question}{Question}{Questions}

\crefname{condition}{Condition}{Conditions}

\crefname{assumption}{Assumption}{Assumptions}

\crefname{propGr}{Effective Property}{Effective Property}

\theoremstyle{remark}

\crefname{rmk}{Remark}{Remarks}
\newtheorem*{rmk*}{Remark}
\crefname{rmk}{Remark}{Remarks}
\newtheorem{remark}[theorem]{Remark}
\crefname{remark}{Remark}{Remarks}

\crefname{fact}{Fact}{Facts}

\crefname{claim}{Claim}{Claims}
\newtheorem*{claim*}{Claim}
\crefname{claim}{Claim}{Claims}

\crefname{step}{Step}{Steps}

\crefname{case}{Case}{Cases}

\numberwithin{equation}{section}
\newcommand{\injmap}{\hookrightarrow}

\newcommand{\lra}{\longrightarrow}

\newcommand{\bC}{\mathbf{C}}

\newcommand{\bF}{\mathbf{F}}

\newcommand{\bK}{\mathbf{K}}

\newcommand{\bQ}{\mathbf{Q}}
\newcommand{\bR}{\mathbf{R}}

\newcommand{\bZ}{\mathbf{Z}}

\newcommand{\bk}{\mathbf{k}}

\newcommand{\sT}{\mathsf{T}}

\newcommand{\A}{\mathsf{A}}

\newcommand{\ch}{\operatorname{char}}

\newcommand{\cl}{\operatorname{cl}}

\newcommand{\dR}{\operatorname{dR}}

\newcommand{\End}{\operatorname{End}}
\newcommand{\et}{{\textrm{\'et}}}

\newcommand{\id}{\operatorname{id}}
\newcommand{\im}{\operatorname{Im}}
\newcommand{\isom}{\simeq}
\newcommand{\Ker}{\operatorname{Ker}}

\newcommand{\N}{\mathsf{N}}

\newcommand{\Tr}{\operatorname{Tr}}

\newcommand{\Z}{\mathsf{Z}}

\begin{document}

\title{An inequality on polarized endomorphisms}

\author{Fei Hu}
\address{Department of Mathematics, University of Oslo, Niels Henrik Abels hus, Moltke Moes vei 35, 0851 Oslo, Norway}
%\email{\href{mailto:hf@u.nus.edu}{\tt hf@u.nus.edu}}
\email{\href{mailto:fhu@math.uio.no}{\tt fhu@math.uio.no}}
%\urladdr{\url{https://sites.google.com/view/feihu90s/}}
\author{Tuyen Trung Truong}
\address{Department of Mathematics, University of Oslo, Niels Henrik Abels hus, Moltke Moes vei 35, 0851 Oslo, Norway}
\email{\href{mailto:tuyentt@math.uio.no}{\tt tuyentt@math.uio.no}}

\begin{abstract}
We show that assuming the standard conjectures, for any smooth projective variety $X$ of dimension $n$ over an algebraically closed field, there is a constant $C>0$ such that for any positive rational number $r$ and for any polarized endomorphism $f$ of $X$, we have
\[
\| G_r \circ f \| \le C \deg(G_r \circ f),
\]
where $G_r$ is a correspondence of $X$ so that for each $0\le i\le 2n$ its pullback action on the $i$-th Weil cohomology group is the multiplication-by-$r^i$ map.
This inequality has been conjectured by the authors to hold in a more general setting, which - in the special case of polarized endomorphisms - confirms the validity of the analog of a well known result by Serre in the K\"ahler setting. 
\end{abstract}


%\date{\today}


\thanks{The authors are supported by Young Research Talents grant \#300814 from the Research Council of Norway.}

\maketitle

\section{Introduction}

Let $X$ be a smooth projective variety of dimension $n$ over an algebraically closed field $\bk$ of arbitrary characteristic and let $H_X$ be a fixed ample divisor on $X$.
Fix a Weil cohomology theory $H^\bullet(X)$ with coefficients in a field $\bF$ of characteristic zero (see \cite[\S 3]{Kleiman94}).
Let $r\in \bQ_{>0}$ be a positive rational number.
Let $\gamma_{r}$ be the homological correspondence of $X$, i.e.,
\[
\gamma_{r} \in H^{2n}(X\times X) \isom \bigoplus_{i=0}^{2n} H^i(X)\otimes_\bF H^{2n-i}(X) \isom \bigoplus_{i=0}^{2n} \End_{\bF}(H^i(X)),
\]
such that its pullback $\gamma_{r}^*$ on $H^i(X)$ is the multiplication-by-$r^i$ map for each $i$.
Note that $\gamma_{r}$ commutes with all homological correspondences of $X$.

If we assume that the standard conjecture $C$ holds on $X$, then the above $\gamma_r$ is algebraic (see \cite[Lemma~4.4]{HT}), and hence is represented by a rational algebraic $n$-cycle $G_r$ on $X\times X$, i.e., $\gamma_r = \cl_{X\times X}(G_r)$.
It is well known that the real vector space $\N^n(X\times X)_\bR$ of numerical cycle classes of codimension $n$ on $X\times X$ is finite dimensional; we thus fix a norm $\|\cdot\|$ on it.
We also fix a degree function $\deg$ on $\N^n(X\times X)_\bR$ with respect to the fixed ample divisor $H_{X\times X} \coloneqq p_1^*H_X + p_2^*H_X$ by setting $\deg(g) \coloneqq g \cdot H_{X\times X}^n$.
The main result of this note is an inequality concerning the norm and the degree of the composite correspondence $G_r\circ f$ of the above $G_r$ and any polarized endomorphism $f$ (viewed as a correspondence), assuming the standard conjectures.
More precisely, we have:

\begin{theorem}
\label{thm:A}
Suppose that the standard conjecture $B$ holds on $X$ and the standard conjecture of Hodge type holds on $X\times X$.
Then for any $r\in \bQ_{>0}$, the above homological correspondence $\gamma_{r}$ of $X$ is algebraic and represented by a rational algebraic $n$-cycle $G_r$ on $X\times X$; moreover, there exists a constant $C>0$ independent of $r$, so that for any polarized endomorphism $f$ of $X$ (i.e., $f^*H_X \sim qH_X$ for some $q\in \bZ_{>0}$), we have
\begin{equation}
\label{eq:A}
\|G_r\circ f\| \leq C \deg(G_r\circ f).
\end{equation}
\end{theorem}

\begin{remark}
(1) Serre \cite{Serre60} proved a result involving eigenvalues of pullbacks on cohomology by polarized endomorphisms of compact K\"ahler manifolds. If the analog of Serre's result holds in positive characteristic, then Weil's Riemann hypothesis follows. 
Grothendieck \cite{Grothendieck69} and Bombieri independently proposed the so-called standard conjectures in order to solve this positive characteristic analog of Serre's result (see also Kleiman's survey articles \cite{Kleiman68,Kleiman94} for details on the standard conjectures). 

It was Deligne \cite{Deligne74} who ingeniously solved Weil's Riemann hypothesis and also generalized it in a form which can be applied to many exponential sums \cite{Deligne80} - the latter being of great interest in analytic number theory. Deligne's proofs, however, are different from what envisioned by Grothendieck. As of today, the standard conjectures, and also the positive characteristic analog of Serre's result, are still widely open. For example, the standard conjecture $D$ is only known in few cases (including the codimension-$1$ case, Abelian varieties \cite{Clozel99}), and the standard conjecture of Hodge type is known only for surfaces and Abelian $4$-folds \cite{Ancona21}. 
%\end{remark}

%\begin{remark}
(2) The authors of this note conjectured in \cite{HT} the inequality \eqref{eq:A} in \cref{thm:A} (in the more general setting of effective correspondences), whose validity will imply, among other things, an earlier conjecture by the second author (see \cite[Question~2]{Truong}). The latter contains as a special case the positive characteristic analog of Serre's result (and hence Weil's Riemann hypothesis as well). We have shown in \cite{HT} that \eqref{eq:A} indeed holds for Abelian varieties (in all dimensions and for all effective correspondences). It also has certain descent properties for generically finite surjective morphisms or even some dominant rational maps, and hence holds for instance for Kummer surfaces. It then has been argued in \cite{HT} that this inequality could be a simpler alternative way in solving the positive characteristic analog of Serre's result considering the difficulty of the standard conjectures. Our \cref{thm:A} confirms that this is indeed the case: for polarized endomorphisms, our inequality \eqref{eq:A} follows from the standard conjectures.
We thus wonder if the general version of the inequality \eqref{eq:A} for effective correspondences is also a consequence of the standard conjectures.
\end{remark}

\section{Proof of Theorem \ref{thm:A}}

Recall that $X$ is a smooth projective variety of dimension $n$ over an algebraically closed field $\bk$ of arbitrary characteristic and $H_X$ is a fixed ample divisor on $X$.
We also fix a Weil cohomology theory $H^\bullet(X)$ with a coefficient field $\bF$ of characteristic zero (see \cite[\S 3]{Kleiman94}).
In particular, we have a cup product $\cup$, Poincar\'e duality, the K\"unneth formula, the cycle class map $\cl_X$, the Lefschetz trace formula, the weak Lefschetz theorem, and the hard Lefschetz theorem.
Examples of classical Weil cohomology theories include:
\begin{itemize}
\item de Rham cohomology $H^\bullet_{\dR}(X(\bC), \bC)$ if $\bk \subseteq \bC$,
\item \'etale cohomology $H^\bullet_{\et}(X, \bQ_\ell)$ with $\ell \neq \ch(\bk)$ if $\bk$ is arbitrary,\footnote{The hard Lefschetz theorem turns out to be very difficult (see \cite[Theorem~4.1.1]{Deligne80}).}
\item crystalline cohomology $H^\bullet_{\rm crys}(X/W(\bk))\otimes \bK$, where $\bK$ is the field of fractions of the Witt ring $W(\bk)$. 
\end{itemize}

For the fixed ample divisor $H_X$ on $X$ and for $0\le i\le 2n-2$, we let
%\[
%L \colon H^i(X) \to H^{i+2}(X), \quad \alpha \mapsto L\alpha \coloneqq \cl_X(H_X) \cup \alpha
%\]
\begin{equation}
\label{eq:Lef}
\begin{array}{ccc}
L \colon H^i(X) &\to& H^{i+2}(X), \\[4pt]
\ \ \alpha &\mapsto& \cl_X(H_X) \cup \alpha
\end{array}
\end{equation}
be the {\it Lefschetz operator}.

By the hard Lefschetz theorem, for any $0\le i\le n$, the $(n-i)$-th iterate $L^{n-i}$ of the Lefschetz operator $L$ is an isomorphism
\[
L^{n-i} \colon H^i(X) \xrightarrow{\ \sim\ } H^{2n-i}(X).
\]
However, $L^{n-i+1} \colon H^i(X) \to H^{2n-i+2}(X)$ may have a nontrivial kernel.
Denote by $P^{i}(X)$ the set of elements $\alpha\in H^i(X)$, called {\it primitive}, satisfying $L^{n-i+1}(\alpha) = 0$, namely,
\begin{equation}
P^{i}(X) \coloneqq \Ker(L^{n-i+1} \colon H^i(X) \to H^{2n-i+2}(X)) \subseteq H^{i}(X).
\end{equation}
This gives us the following primitive decomposition (a.k.a. Lefschetz decomposition):
\begin{equation}
H^i(X) = \bigoplus_{j\ge i_0} L^j P^{i-2j}(X),
\end{equation}
where $i_0 \coloneqq \max(i-n,0)$.

\begin{definition}[{cf.~\cite[\S 1.4]{Kleiman68}}]
For any $\alpha \in H^i(X)$, we write
\begin{equation}
\label{eq:primitive}
\alpha = \sum_{j\ge i_0} L^j(\alpha_j), \quad \alpha_j \in P^{i-2j}(X).
\end{equation}
Then we define an operator $*$ as follows:
%\begin{align*}
%* \colon H^i(X) &\to H^{2n-i}(X), \\
%\alpha &\mapsto *\alpha \coloneqq \sum_{j\ge i_0} (-1)^{\frac{(i-2j)(i-2j+1)}{2}} L^{n-i+j}(\alpha_j).
%\end{align*}
\begin{equation}
\label{eq:star}
\begin{array}{ccc}
* \colon H^i(X) &\to& H^{2n-i}(X), \\[4pt]
\ \ \alpha &\mapsto& \displaystyle \sum_{j\ge i_0} (-1)^{\frac{(i-2j)(i-2j+1)}{2}} L^{n-i+j}(\alpha_j).
\end{array}
\end{equation}
\end{definition}
It is easy to check that $*^2=\id$.
The standard conjecture $B(X)$ predicts that the above homological correspondence $*$ is algebraic (cf.~\cite[Proposition~2.3]{Kleiman68}).
%We refer to \cite[\S 3]{Kleiman68} for the precise statement of the standard conjecture of Hodge type.

For any correspondence $g$ of $X$, denote by $g'$ its adjoint with respect to the following non-degenerate bilinear form
\begin{equation}
\label{eq:Weil-form}
\begin{array}{ccc}
H^i(X) \times H^i(X) & \lra & \bF \\[4pt]
(\alpha, \beta) & \mapsto & \langle \alpha, \beta \rangle \coloneqq \alpha \cup *\beta.
\end{array}
\end{equation}
In other words, we have $g' = *\circ g^\sT \circ *$ by definition, where $g^\sT$ denotes the canonical transpose of $g$ by interchanging the coordinates.

For any $0\le k\le n$, let $\A^k(X)\subseteq H^{2k}(X)$ denote the $\bQ$-vector space of cohomology classes generated by algebraic cycles of codimension $k$ on $X$ under the cycle class map $\cl_X$, i.e.,
\[
\A^k(X) \coloneqq \im(\cl_X \colon \Z^k(X)_\bQ \lra H^{2k}(X)).
\]
The standard conjecture of Hodge type predicts that when restricted to $\A^k(X)$ the bilinear form \eqref{eq:Weil-form} is positive definite (see \cite[\S 3]{Kleiman68} for details).

\begin{lemma}
\label{lemma:A}
Let $\pi_i \in H^i(X) \otimes H^{2n-i}(X)$ be the $i$-th K\"unneth component of the diagonal class, which corresponds to the projection operator $H^\bullet(X)\to H^i(X)$ via the pullback.
Then for any polarized endomorphism $f$ of $X$ (i.e., $f^*H_X \sim qH_X$ for some $q\in \bZ_{>0}$), we have
\[
(\pi_i \circ f) \circ (\pi_i \circ f)' = q^i \pi_i
\]
as homological correspondences.
\end{lemma}
\begin{proof}
Note that for any $\alpha \in H^i(X)$ with the above primitive decomposition \eqref{eq:primitive},
\[
\sum_{j\ge i_0} L^j(q^j f^*\alpha_j) \text{ with } f^*(\alpha_j) \in P^{i-2j}(X)
\]
is the primitive decomposition of $f^*(\alpha)$.
It follows that
\begin{align*}
((\pi_i \circ f) \circ (\pi_i \circ f)')^*(\alpha) &= *\circ (\pi_i \circ f)_* \circ * \circ (\pi_i \circ f)^*(\alpha) \\
&= * \circ (\pi_i \circ f)_* \circ * \circ f^*(\alpha) \\
&= * \circ \pi_{2n-i}^* \circ f_*\sum_{j\ge i_0} (-1)^{\frac{(i-2j)(i-2j+1)}{2}} L^{n-i+j}(q^j f^*\alpha_j) \\
&= * \sum_{j\ge i_0} (-1)^{\frac{(i-2j)(i-2j+1)}{2}} f_*(H_X^{n-i+j} \cup q^j f^*\alpha_j) \\
&= * \sum_{j\ge i_0} (-1)^{\frac{(i-2j)(i-2j+1)}{2}} f_*H_X^{n-i+j} \cup q^j \alpha_j \\
&= * \sum_{j\ge i_0} (-1)^{\frac{(i-2j)(i-2j+1)}{2}} q^{i-j} H_X^{n-i+j} \cup q^j \alpha_j \\
&= * \sum_{j\ge i_0} (-1)^{\frac{(i-2j)(i-2j+1)}{2}} q^i L^{n-i+j}(\alpha_j) \\
&= q^i *^2 \alpha = q^i \alpha,
\end{align*}
where $\pi_i^*$ and $(\pi_i)_* = \pi_{2n-i}^*$ are projections to $H^i(X)$ and $H^{2n-i}(X)$, respectively, the third equality follows from the definition of the $*$ operator, the fifth one follows from the projection formula, and the last one follows from the fact that $*^2 = \id$.
This yields the lemma.
\end{proof}

\begin{proof}[Proof of Theorem~\ref{thm:A}]
Since the standard conjecture $B(X)$ implies the standard conjecture $C(X)$, the algebraicity of $\gamma_r$ follows from \cite[Lemma~4.4]{HT}.
Also, by assumption, the bilinear form \eqref{eq:Weil-form} is a Weil form; see \cite[Theorem~3.11]{Kleiman68}.
In particular, if we let $\Delta_i\in \Z^n(X\times X)_\bQ$ represent $\pi_i$ and let $f_i$ denote the composite correspondence $\Delta_i \circ f$, then the square root of
\[
\Tr((f_i \circ f_i')^*|_{H^\bullet(X)}) = \Tr((f_i \circ f_i')^*|_{H^i(X)}) \in \bQ_{>0}
\]
gives us a norm of $f^*|_{H^i(X)}$.
On the other hand, it follows from \cref{lemma:A} that
\[
\Tr((f_i \circ f_i')^*|_{H^i(X)}) = q^i b_i(X),
\]
where $b_i(X) \coloneqq \dim_\bF H^i(X)$ is the $i$-th Betti number of $X$.
Putting together, we thus obtain that
\[
\big\|f^*|_{H^i(X)}\big\| = \sqrt{b_i(X)} \, q^{i/2}.
\]

Now, we let $g$ denote $G_r\circ f$.
By assumption, the standard conjecture $D$ holds on $X\times X$ (see \cite[Corollaries~3.9, 2.5, and 2.2]{Kleiman68}) and hence the cycle class map induces an injective map
\[
\N^n(X\times X) \otimes_\bZ \bF \injmap H^{2n}(X\times X).
\]
It thus follows that
\[
\|g\| \lesssim \|\cl_{X\times X}(g)\|.
\]
Here the right-hand side denotes a norm on $H^{2n}(X\times X) \isom \bigoplus_{i=0}^{2n} \End_{\bF}(H^i(X))$ which is equivalent to
\[
\max_{0\le i\le 2n} \big\|g^*|_{H^i(X)}\big\|.
\]
Note that the above equivalence part depends on the choices of norms.
Also, by the definitions of $G_r$ and $f$ we have that $g^*|_{H^i(X)} = r^i f^*|_{H^i(X)}$ and
\[
\deg(g) = g \cdot H_{X\times X}^n = \sum_{k=0}^n \binom{n}{k} \deg_k(g) = \sum_{k=0}^n \binom{n}{k} r^{2k} q^k.
\]
If $i=2k$ is even, then we have
\[
\big\|g^*|_{H^i(X)}\big\| = r^{2k} \big\|f^*|_{H^{2k}(X)}\big\| = r^{2k} \sqrt{b_{2k}(X)} \, q^{k} = \sqrt{b_{2k}(X)} \deg_k(g).
\]
When $i=2k+1$ is odd, similarly, we also have
\begin{align*}
\big\|g^*|_{H^i(X)}\big\| &= r^{2k+1} \big\|f^*|_{H^{2k+1}(X)}\big\| \\
&= r^{2k+1} \sqrt{b_{2k+1}(X) \, q^{2k+1}} \\
&\le \sqrt{b_{2k+1}(X)} \, (r^{2k} q^k + r^{2k+2} q^{k+1})/2 \\
&\le \sqrt{b_{2k+1}(X)} \, \max\{r^{2k} q^k, r^{2k+2} q^{k+1}\} \\
&\le \sqrt{b_{2k+1}(X)} \, \max\{\deg_k(g), \deg_{k+1}(g)\}.
\end{align*}
So overall, there is a constant $C$ depending only on the Betti numbers $b_i(X)$ of $X$, the dimension $n$ of $X$, and norms we have chosen, but independent of $f$ nor $r$, such that $\|g\| \le C \deg(g)$.
We thus prove \cref{thm:A}.
\end{proof}

%\linespread{1.1}

%\bibliographystyle{amsplain}
\bibliographystyle{amsalpha}
\bibliography{../mybib}

\end{document}